\documentclass[oneside]{amsart}

\usepackage{amsthm}
\usepackage{float}

\usepackage{geometry}                		
\usepackage[utf8]{inputenc}
\usepackage[english]{babel}
\usepackage{tikz}

\theoremstyle{definition}

\newtheorem{thm}{Theorem}[section]
\newtheorem{defn}[thm]{Definition}
\newtheorem{cor}[thm]{Corollary}

\newtheorem{rmk}[thm]{Remark}
\newtheorem{ex}[thm]{Example}

\theoremstyle{definition}

 \usepackage{amsmath,amssymb,amsthm}

\begin{document}
	
	\title{A tiling proof of Euler's Pentagonal Number Theorem and generalizations}
	
	\author{Dennis Eichhorn}
	
	\author{Hayan Nam}
	
	\author{Jaebum Sohn}
	
	
	\thanks{The third named author's work was supported by the National Research Foundation of Korea (NRF) NRF-2017R1A2B4009501.}

\begin{abstract}
	In two papers, Little and Sellers introduced an exciting new combinatorial method for proving partition identities which is not directly bijective.
	Instead, they consider various sets of weighted tilings of a $1 \times \infty$ board with squares and dominoes, and for each type of tiling they construct a generating function in two different ways, which generates a $q$-series identity.
	Using this method, they recover quite a few classical $q$-series identities, but Euler's Pentagonal Number Theorem is not among them.
	In this paper, we introduce a key parameter when constructing the generating functions of various sets of tilings which allows us to recover Euler's Pentagonal Number Theorem along with an infinite family of generalizations.
	\end{abstract}	
	
	\maketitle
		
\section{Introduction}
	
	Proofs of $q$-series identities generally fall into one of three categories:  proofs by classical by $q$-series manipulations, proofs appealing to modular forms, or proofs using bijective methods showing that two sets of combinatorial objects generated by the two sides of an identity are equinumerous. 
	In \cite{LS1} and \cite{LS2},  Little and Sellers introduced a new combinatorial method for proving partition identities which is not directly bijective.
Instead, they consider various sets of weighted tilings of a $1 \times \infty$ board with squares and dominoes, and for each type of tiling they construct a generating function for the set of all tilings in two different ways.
Since this creates two different expressions for the same object, it generates a $q$-series identity.
Using this method, they recover quite a few classical $q$-series identities, 
ranging from identities like
$$
\sum_{n \geq 0} \frac{(-z;q)_n}{(q^2;q^2)_n}q^{n^2+n} 
= \prod_{n\geq 1} (1+q^{4n-2})(1+zq^{4n-2})(1+q^{4n})
$$
due to G\"ollnitz
to more complicated identities such as
$$
\sum_{n = 0}^{\infty} \frac{z^n q^{n^2+n} (1-z^2q^{2n+3})}{(q;q)_n} 
= (-zq;q)_{\infty} \sum_{n = 0}^{\infty} 
    \frac{(-1)^n z^{2n} q^{3n^2}}{(q^2;q^2)_n (-zq;q)_{2n+1}},
$$
where we are using the usual $q$-series notation that defines 
$$(a;q)_{n} = \prod_{i=0}^{n-1} (1-aq^i) \qquad \text{and} \qquad  
(a;q)_{\infty} = \lim_{n \to \infty} (a;q)_{n}.
$$
Interestingly, 
Euler's Pentagonal Number Theorem is not among the results they recover using their method.

In this paper, we modify their method of generating $q$-series identities by introducing a new statistic on tilings, the rank.
By considering the rank of each tiling, we construct our generating functions in a new way
which allows us to recover Euler's Pentagonal Number Theorem along with a
one-parameter generalization.
Fortunately, the notion of the rank of a tiling lends itself to a two-parameter generalization, the $(k,\ell)$-rank.  
This leads directly to a much more general identity which is essentially a three-parameter generalization of Euler's Pentagonal Number Theorem.


\section{Tilings and the rank}\label{tileandrank}

As in \cite{LS1} and \cite{LS2}, we consider tiling a $1 \times \infty$ board 
with tiles of various colors.
For the bulk of the paper, we will only consider tilings with
white squares and black squares where each tiling has a finite number of black squares.
We define the \emph{position} of a square to be its location on the board, a positive integer, and every tiling has exactly one square at each position.
We define the \emph{weight} of a tile $t$ to be 
$$
w(t) =  \begin{cases} zq^i & \text{if $t$ is a black square in position $i$,} \\
1 & \text{if $t$ is a white square in position $i$,} \end{cases}
$$
and the weight of a tiling $T$ is
$$
w(T) = \prod_{t \in T} w(t),
$$
the product of the weights of its tiles.
Notice if a tiling $T$ has 
$m$ black tiles and 
$n$ is the sum of the positions of the black tiles in the tiling,
then the weight of the tiling is $w(T) = z^m q^n$.

\begin{ex} \label{example1}
	Consider the tiling with black squares at positions
	$3,4,6,7,8,11,12,13,14,15,16$, and $18$.
	Since the white squares at the other positions have weight $1$, they do not contribute to the weight of the tiling. 
	Thus the weight of the tiling is the product of the weights of the black squares, which is 
	$$
	 zq^{3} \times zq^{4} \times zq^{6} \times zq^{7} \times zq^{8} \times zq^{11} \times zq^{12} \times zq^{13} \times zq^{14} \times zq^{15} \times zq^{16} \times zq^{18}  = z^{12}q^{127}. 
	$$
\end{ex}

\begin{figure}[h]
	
	\centering
	
	\begin{tikzpicture}
	
	\draw [thick] (10,0) -- (0,0) -- (0,0.5) -- (10,0.5);
	
	\node at (0.25,-0.25) {1};
	
	\node at (0.75,-0.25) {2};
	
	\node at (1.25,-0.25) {3};
	
	\node at (1.75,-0.25) {4};
	
	\node at (2.25,-0.25) {5};
	
	\node at (2.75,-0.25) {6};
	
	\node at (3.25,-0.25) {7};
	
	\node at (3.75,-0.25) {8};
	
	\node at (4.25,-0.25) {9};
	
	\node at (4.72,-0.25) {10};
	
	\node at (5.22,-0.25) {11};
	
	\node at (5.72,-0.25) {12};
	
	\node at (6.22,-0.25) {13};
	
	\node at (6.72,-0.25) {14};
	
	\node at (7.22,-0.25) {15};
	
	\node at (7.72,-0.25) {16};
	
	\node at (8.22,-0.25) {17};
	
	\node at (8.72,-0.25) {18};
	
	\node at (9.22,-0.25) {19};
	
	\node at (9.76,-0.25) {$\cdots$};
	
	\draw [thick] (0.5,0) -- (0.5,0.5);
	
	\draw [thick] (1,0) -- (1,0.5);
	
	\draw [fill=black] (1,0) rectangle (1.5,0.5);
	
	\draw [thick, white] (1.5,0) -- (1.5,0.5);
	
	\draw [fill=black] (1.5,0) rectangle (2,0.5);
	
	\draw [thick] (2,0) -- (2,0.5);
	
	\draw [fill=black] (2.5,0) rectangle (3,0.5);
	
	\draw [thick] (2.5,0) -- (2.5,0.5);
	
	\draw [fill=black] (3,0) rectangle (3.5,0.5);
	
	\draw [thin, white] (3,0) -- (3,0.5);
	
	\draw [fill=black] (3.5,0) rectangle (4,0.5);
	
	\draw [thin, white] (3.5,0) -- (3.5,0.5);
	
	\draw [thick] (4,0) -- (4,0.5);
	
	\draw [thick] (4.5,0) -- (4.5,0.5);
	
	\draw [thick] (5,0) -- (5,0.5);
	
	\draw [fill=black] (5,0) rectangle (5.5,0.5);
	
	\draw [thick, white] (5.5,0) -- (5.5,0.5);
	
	\draw [fill=black] (5.5,0) rectangle (6,0.5);
	
	\draw [thick, white] (6,0) -- (6,0.5);
	
	\draw [fill=black] (6,0) rectangle (6.5,0.5);
	
	\draw [thick, white] (6.5,0) -- (6.5,0.5);
	
	\draw [fill=black] (6.5,0) rectangle (7,0.5);
	
	\draw [thick, white] (7,0) -- (7,0.5);
	
	\draw [fill=black] (7,0) rectangle (7.5,0.5);
	
	\draw [thick, white] (7.5,0) -- (7.5,0.5);
	
	\draw [fill=black] (7.5,0) rectangle (8,0.5);
	
	\draw [thick] (8,0) -- (8,0.5);
	
	\draw [thick] (8.5,0) -- (8.5,0.5);
	
	\draw [fill=black] (8.5,0) rectangle (9,0.5);
	
	\draw [thick] (9,0) -- (9,0.5);
	
	\draw [thick] (9.5,0) -- (9.5,0.5);
	
	\end{tikzpicture}
	
	\caption{A tiling with black squares at positions $3,4,6,7,8,11,12,13,14,15,16$, and $18$.}
	
\end{figure}
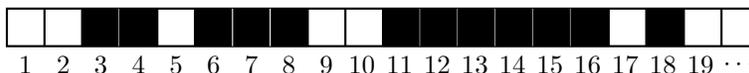

Define 
$$
F(z, q) = \sum_{\text{tilings} \ T} w(T).
$$
We call $F(z, q)$ the \emph{generating function} of all tilings of the $1 \times \infty$ board. 
Notice that the contribution to the total weight of the tiling of the tile at position $i$ must be either $1$ or $zq^i$, and thus
\begin{equation}\label{simpleF}
F(z, q) = \prod_{i=1}^{\infty} (1+zq^i) = (-zq;q)_{\infty}.
\end{equation}

Since tilings are determined completely by the positions of their black tiles,
we will often only discuss the positions of black tiles.
Some of our treatment of tilings will involve changing the positions of various tiles;
it is to be assumed that whenever we move black tiles,
any displaced white tiles are simply removed, and any vacancies left by moving 
black tiles are filled in with white tiles as needed to complete the tiling.
Little and Sellers \cite{LS2} defined one particular moving process which they called the \emph{projection} of tiles, which we adopt here.
Define projection to be moving some number of black tiles to the right in a way that does not change the relative positions of the black tiles.
Observe that projecting one black tile $k$ positions to the right increases the weight of a tiling by a factor of $q^k$.

We are now ready to recompute $F(z, q)$ in a new way.
Notice that every tiling with $m$ black tiles could be constructed by first
placing black tiles in positions $1$ through $m$ (with white tiles at all other positions)
and then projecting those $m$ black tiles out to their proper positions.
The weight of the initial placement is $z^m q^{m(m+1)/2}$, and 
there are several ways in which we could group tiles as we project them from positions $1$ through $m$ out to any $m$ given positions $p_1 < p_2 < \dots < p_m$.
One method is to first project all $m$ tiles to the right until the leftmost tile is in position $p_1$,
then project the $m-1$ remaining tiles to the right until the second leftmost tile is in position $p_2$, 
and then iteratively continuing to project the remaining tiles to the right until 
the tiles are in positions $p_1$ through $p_m$. 
Projecting $m$ tiles $k$ positions to the right will increase the weight by a factor of $q^{mk}$.
To generate any possible tiling, we must allow for projections every possible nonnegative distance, and so generating the first projection of all $m$ tiles introduces a factor of 
$(1+q^m+q^{2m}+q^{3m}+ \cdots ) = 1/(1-q^m)$
to the generating function.
Similarly, the projection of the next $m-1$ tiles 
introduces a factor of 
$(1+q^{m-1}+q^{2(m-1)}+q^{3(m-1)}+ \cdots ) = 1/(1-q^{m-1})$, and in general,
the projection of the next $m-i$ tiles 
introduces a factor of 
$(1+q^{m-i}+q^{2(m-i)}+q^{3(m-i)}+ \cdots ) = 1/(1-q^{m-i})$.
Summing over every possible number of black tiles $m$,
we have 
\begin{equation} \label{projectedF}
F(z, q) = \sum_{m=0}^{\infty} \frac{z^m q^{m(m+1)/2}}{(q;q)_{m}}.
\end{equation}

Equating our two expressions for $F(z, q)$ in (\ref{simpleF}) and (\ref{projectedF}),
we recover an identity of Euler \cite{Euler}, namely 
\begin{equation}
(-zq;q)_{\infty} = \sum_{m=0}^{\infty} \frac{z^m q^{m(m+1)/2}}{(q;q)_{m}}.
\end{equation}

This is perhaps the simplest example of a $q$-series identity that can be proven using tilings and projection. 
Little and Sellers actually start with more sophisticated identities by treating 
tilings with squares and dominoes.


In the next subsection, we introduce the rank of a tiling, which is the critical
statistic that will allow us to use this same method to recover Euler's Pentagonal Number Theorem
and a one-parameter generalization due to Sylvester.

\subsection{The rank}

Indexing tilings by their ranks instead of their total number of black squares gives us 
a new way to prove $q$-series identities.
For a given tiling, let $b(m)$ be the number of black squares at positions greater than $m$. 

\begin{defn} \label{def:rank}
The \emph{rank} of a tiling is the least $m$ such that $b(m) \leq m$.
\end{defn}

\begin{ex} \label{example2}
	Consider the tiling with black squares at positions
	$3,4,6,7,8,11,12,13,14,15,16$, and $18$.
	For each position $m$, we compute $b(m)$ in Table \ref{btable}.
	Since $b(7) = 8$ and $b(8) = 7$, we see that the rank of this tiling is $8$.
\end{ex}

\begin{table}[h] 
	\begin{center}
		\begin{tabular}{ l|*{20}{c}r}
			\hline
			$m$ & 1 & 2 & 3 & 4 & 5 & 6 & 7 & 8 & 9 & 10 & 11 & 12 & 13 & 14 & 15 & 16 & 17 & 18 & 19 \\  \hline
			$b(m)$ & 12 & 12 & 11 & 10 & 10 & 9 & 8 & 7 & 7 & 7 & 6 & 5 & 4 & 3 & 2 & 1 & 1 & 0 & 0 \\ \hline
		\end{tabular}
	\end{center}
	\caption{$b(m)$ is the number of black squares at positions greater than $m$.}
	\label{btable}
\end{table}

Notice that changing the colors of any of the tiles in the first seven position of the tiling above does not change the rank.
The tiling with black squares at 
$1,2,3,4,5,6,7,8,11,12,13,14,15,16$, and $18$ also has rank $8$,
as does the tiling with black squares at just
$8,11,12,13,14,15,16$, and $18$, for example (see Figure \ref{rank8exs}).

\begin{figure}[h]
	\centering
	\begin{tikzpicture}
	\draw [thick] (10,0) -- (0,0) -- (0,0.5) -- (10,0.5);
	\node at (0.25,-0.25) {1};
	\node at (0.75,-0.25) {2};
	\node at (1.25,-0.25) {3};
	\node at (1.75,-0.25) {4};
	\node at (2.25,-0.25) {5};
	\node at (2.75,-0.25) {6};
	\node at (3.25,-0.25) {7};
	\node at (3.75,-0.25) {8};
	\node at (4.25,-0.25) {9};
	\node at (4.72,-0.25) {10};
	\node at (5.22,-0.25) {11};
	\node at (5.72,-0.25) {12};
	\node at (6.22,-0.25) {13};
	\node at (6.72,-0.25) {14};
	\node at (7.22,-0.25) {15};
	\node at (7.72,-0.25) {16};
	\node at (8.22,-0.25) {17};
	\node at (8.72,-0.25) {18};
	\node at (9.22,-0.25) {19};
	\node at (9.76,-0.25) {$\cdots$};
	\draw [thick] (0.5,0) -- (0.5,0.5);
	\draw [thick] (1,0) -- (1,0.5);
	\draw [fill=black] (1,0) rectangle (1.5,0.5);
	\draw [thick, white] (1.5,0) -- (1.5,0.5);
	\draw [fill=black] (1.5,0) rectangle (2,0.5);
	\draw [thick] (2,0) -- (2,0.5);
	\draw [fill=black] (2.5,0) rectangle (3,0.5);
	\draw [thick] (2.5,0) -- (2.5,0.5);
	\draw [fill=black] (3,0) rectangle (3.5,0.5);
	\draw [thin, white] (3,0) -- (3,0.5);
	\draw [fill=black] (3.5,0) rectangle (4,0.5);
	\draw [thin, white] (3.5,0) -- (3.5,0.5);
	\draw [thick] (4,0) -- (4,0.5);
	\draw [thick] (4.5,0) -- (4.5,0.5);
	\draw [thick] (5,0) -- (5,0.5);
	\draw [fill=black] (5,0) rectangle (5.5,0.5);
	\draw [thick, white] (5.5,0) -- (5.5,0.5);
	\draw [fill=black] (5.5,0) rectangle (6,0.5);
	\draw [thick, white] (6,0) -- (6,0.5);
	\draw [fill=black] (6,0) rectangle (6.5,0.5);
	\draw [thick, white] (6.5,0) -- (6.5,0.5);
	\draw [fill=black] (6.5,0) rectangle (7,0.5);
	\draw [thick, white] (7,0) -- (7,0.5);
	\draw [fill=black] (7,0) rectangle (7.5,0.5);
	\draw [thick, white] (7.5,0) -- (7.5,0.5);
	\draw [fill=black] (7.5,0) rectangle (8,0.5);
	\draw [thick] (8,0) -- (8,0.5);
	\draw [thick] (8.5,0) -- (8.5,0.5);
	\draw [fill=black] (8.5,0) rectangle (9,0.5);
	\draw [thick] (9,0) -- (9,0.5);
	\draw [thick] (9.5,0) -- (9.5,0.5);
	\end{tikzpicture}

	\begin{tikzpicture}
	\draw [thick] (10,0) -- (0,0) -- (0,0.5) -- (10,0.5);
	\node at (0.25,-0.25) {1};
	\node at (0.75,-0.25) {2};
	\node at (1.25,-0.25) {3};
	\node at (1.75,-0.25) {4};
	\node at (2.25,-0.25) {5};
	\node at (2.75,-0.25) {6};
	\node at (3.25,-0.25) {7};
	\node at (3.75,-0.25) {8};
	\node at (4.25,-0.25) {9};
	\node at (4.72,-0.25) {10};
	\node at (5.22,-0.25) {11};
	\node at (5.72,-0.25) {12};
	\node at (6.22,-0.25) {13};
	\node at (6.72,-0.25) {14};
	\node at (7.22,-0.25) {15};
	\node at (7.72,-0.25) {16};
	\node at (8.22,-0.25) {17};
	\node at (8.72,-0.25) {18};
	\node at (9.22,-0.25) {19};
	\node at (9.76,-0.25) {$\cdots$};
	\draw [thick] (0.5,0) -- (0.5,0.5);
	\draw [thick] (1,0) -- (1,0.5);
	\draw [thick] (1.5,0) -- (1.5,0.5);
	\draw [thick] (2,0) -- (2,0.5);
	\draw [thick] (2.5,0) -- (2.5,0.5);
	\draw [thick] (3,0) -- (3,0.5);
	\draw [fill=black] (3.5,0) rectangle (4,0.5);
	\draw [thin, white] (3.5,0) -- (3.5,0.5);
	\draw [thick] (4,0) -- (4,0.5);
	\draw [thick] (4.5,0) -- (4.5,0.5);
	\draw [thick] (5,0) -- (5,0.5);
	\draw [fill=black] (5,0) rectangle (5.5,0.5);
	\draw [thick, white] (5.5,0) -- (5.5,0.5);
	\draw [fill=black] (5.5,0) rectangle (6,0.5);
	\draw [thick, white] (6,0) -- (6,0.5);
	\draw [fill=black] (6,0) rectangle (6.5,0.5);
	\draw [thick, white] (6.5,0) -- (6.5,0.5);
	\draw [fill=black] (6.5,0) rectangle (7,0.5);
	\draw [thick, white] (7,0) -- (7,0.5);
	\draw [fill=black] (7,0) rectangle (7.5,0.5);
	\draw [thick, white] (7.5,0) -- (7.5,0.5);
	\draw [fill=black] (7.5,0) rectangle (8,0.5);
	\draw [thick] (8,0) -- (8,0.5);
	\draw [thick] (8.5,0) -- (8.5,0.5);
	\draw [fill=black] (8.5,0) rectangle (9,0.5);
	\draw [thick] (9,0) -- (9,0.5);
	\draw [thick] (9.5,0) -- (9.5,0.5);
	\end{tikzpicture}

	\begin{tikzpicture}
	\draw [thick] (10,0) -- (0,0) -- (0,0.5) -- (10,0.5);
	\node at (0.25,-0.25) {1};
	\node at (0.75,-0.25) {2};
	\node at (1.25,-0.25) {3};
	\node at (1.75,-0.25) {4};
	\node at (2.25,-0.25) {5};
	\node at (2.75,-0.25) {6};
	\node at (3.25,-0.25) {7};
	\node at (3.75,-0.25) {8};
	\node at (4.25,-0.25) {9};
	\node at (4.72,-0.25) {10};
	\node at (5.22,-0.25) {11};
	\node at (5.72,-0.25) {12};
	\node at (6.22,-0.25) {13};
	\node at (6.72,-0.25) {14};
	\node at (7.22,-0.25) {15};
	\node at (7.72,-0.25) {16};
	\node at (8.22,-0.25) {17};
	\node at (8.72,-0.25) {18};
	\node at (9.22,-0.25) {19};
	\node at (9.76,-0.25) {$\cdots$};
	\draw [fill=black] (0,0) rectangle (0.5,0.5);
	\draw [thick, white] (0.5,0) -- (0.5,0.5);
	\draw [fill=black] (0.5,0) rectangle (1,0.5);
	\draw [ultra thick, white] (1,0) -- (1,0.5);
	\draw [fill=black] (1,0) rectangle (1.5,0.5);
	\draw [thick, white] (1.5,0) -- (1.5,0.5);
	\draw [fill=black] (1.5,0) rectangle (2,0.5);
	\draw [thick, white] (2,0) -- (2,0.5);
	\draw [fill=black] (2,0) rectangle (2.5,0.5);
	\draw [fill=black] (2.5,0) rectangle (3,0.5);
	\draw [thin, white] (2.5,0) -- (2.5,0.5);
	\draw [fill=black] (3,0) rectangle (3.5,0.5);
	\draw [thin, white] (3,0) -- (3,0.5);
	\draw [fill=black] (3.5,0) rectangle (4,0.5);
	\draw [thin, white] (3.5,0) -- (3.5,0.5);
	\draw [thick] (4,0) -- (4,0.5);
	\draw [thick] (4.5,0) -- (4.5,0.5);
	\draw [thick] (5,0) -- (5,0.5);
	\draw [fill=black] (5,0) rectangle (5.5,0.5);
	\draw [thick, white] (5.5,0) -- (5.5,0.5);
	\draw [fill=black] (5.5,0) rectangle (6,0.5);
	\draw [thick, white] (6,0) -- (6,0.5);
	\draw [fill=black] (6,0) rectangle (6.5,0.5);
	\draw [thick, white] (6.5,0) -- (6.5,0.5);
	\draw [fill=black] (6.5,0) rectangle (7,0.5);
	\draw [thick, white] (7,0) -- (7,0.5);
	\draw [fill=black] (7,0) rectangle (7.5,0.5);
	\draw [thick, white] (7.5,0) -- (7.5,0.5);
	\draw [fill=black] (7.5,0) rectangle (8,0.5);
	\draw [thick] (8,0) -- (8,0.5);
	\draw [thick] (8.5,0) -- (8.5,0.5);
	\draw [fill=black] (8.5,0) rectangle (9,0.5);
	\draw [thick] (9,0) -- (9,0.5);
	\draw [thick] (9.5,0) -- (9.5,0.5);
	\end{tikzpicture}
	\caption{Several tilings with rank 8.}
	\label{rank8exs}
\end{figure}
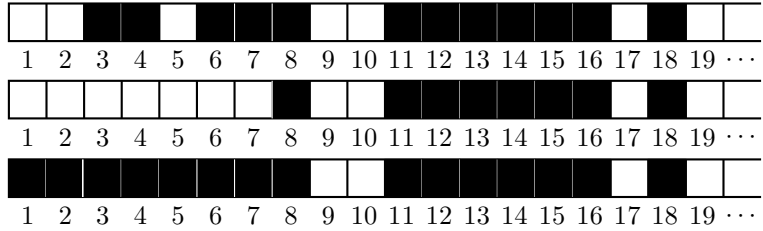

\begin{rmk} \label{rankcases}
	Notice that $b(m)$ is non-increasing, and that $b(m-1)-b(m) \leq 1$.
That being the case, we see that the rank of a tiling will be
the unique $m$ such that either
\begin{enumerate}
	\item $b(m) = m$, or
	\item $b(m)=m-1$ and $b(m-1)=m$,
\end{enumerate}
	and these two cases are mutually exclusive.
This will allow us to take a sum over all tilings by summing these two cases over every $m$ in our results below.
\end{rmk}

Conceptually, the reader may find it helpful to think of the graph $b(m)$ as a path along lattice points.
Remark \ref{rankcases} reflects the fact that 
every step of the lattice path is either directly to the right or at a slope of $-1$, and
as this lattice path goes from the vertical axis down to the horizontal axis, there are two possible cases.
Either (1), it will pass through a lattice point where 
$b(m) = m$, or (2), it will pass through a pair of lattice points such that
$b(m-1)=m$ and $b(m)=m-1$.
In Figure \ref{fig:M1} below,
we have indicated this by graphing the points $(m,m)$ in gray along with the line segments connecting 
$(m-1,m)$ and $(m,m-1)$ for each $m$.
For every tiling, the corresponding graph of $b(m)$ will either intersect one of the points,
or it will traverse one of the line segments (but not both).


\begin{figure}[H]
	\centering
	\begin{tikzpicture}
	\draw  [<->, thick](0,7) -- (0,0) -- (10,0);
	\node at (11,0) {position};
	\node at (0,7.5) {$b(m)$};
	\node at (-0.25,-0.25) {0};
	\node at (0,0.5) {---};
	\node at (0,1) {---};
	\node at (0,1.5) {---};
	\node at (0,2) {---};
	\node at (0,2.5) {---};
	\node at (-0.3,2.5) {5};
	\node at (0,3) {---};
	\node at (0,3.5) {---};
	\node at (0,4) {---};
	\node at (0,4.5) {---};
	\node at (0,5) {---};
	\node at (-0.35,5) {10};
	\node at (0,5.5) {---};
	\node at (0,6) {---};
	\node at (0,6.5) {---};
	\node at (0.5,0) {$\mid$};
	\node at (1,0) {$\mid$};
	\node at (1.5,0) {$\mid$};
	\node at (2,0) {$\mid$};
	\node at (2.5,0) {$\mid$};
	\node at (2.5,-0.3) {5};
	\node at (3,0) {$\mid$};
	\node at (3.5,0) {$\mid$};
	\node at (4,0) {$\mid$};
	\node at (4.5,0) {$\mid$};
	\node at (5,0) {$\mid$};
	\node at (5,-0.3) {10};
	\node at (5.5,0) {$\mid$};
	\node at (6,0) {$\mid$};
	\node at (6.5,0) {$\mid$};
	\node at (7,0) {$\mid$};
	\node at (7.5,0) {$\mid$};
	\node at (7.5,-0.3) {15};
	\node at (8,0) {$\mid$};
	\node at (8.5,0) {$\mid$};
	\node at (9,0) {$\mid$};
	\node at (9.5,0) {$\mid$};

	\draw [fill] (-0.07,5.93) rectangle (0.07,6.07);
	\draw [fill] (0.43,5.93) rectangle (0.57,6.07);
	\draw [fill] (0.93,5.93) rectangle (1.07,6.07);
	\draw [fill] (1.43,5.43) rectangle (1.57,5.57);
	\draw [fill] (1.93,4.93) rectangle (2.07,5.07);
	\draw [fill] (2.43,4.93) rectangle (2.57,5.07);
	\draw [fill] (2.93,4.43) rectangle (3.07,4.57);
	\draw [fill] (3.43,3.93) rectangle (3.57,4.07);
	\draw [fill] (3.93,3.43) rectangle (4.07,3.57);
	\draw [fill] (4.43,3.43) rectangle (4.57,3.57);	
	\draw [fill] (4.93,3.43) rectangle (5.07,3.57);
	\draw [fill] (5.43,2.93) rectangle (5.57,3.07);
	\draw [fill] (5.93,2.43) rectangle (6.07,2.57);
	\draw [fill] (6.43,1.93) rectangle (6.57,2.07);
	\draw [fill] (6.93,1.43) rectangle (7.07,1.57);
	\draw [fill] (7.43,0.93) rectangle (7.57,1.07);
	\draw [fill] (7.93,0.43) rectangle (8.07,0.57);
	\draw [fill] (8.43,0.43) rectangle (8.57,0.57);
	\draw [fill] (8.93,-0.07) rectangle (9.07,0.07);
	\draw [fill] (9.43,-0.07) rectangle (9.57,0.07);
	
	\filldraw[thick, gray] (0,0) circle [radius=0.07];
	\draw (0,0.5) circle [radius=0.07];
	\draw (0.5,0) circle [radius=0.07];
	\draw[thick] (0,0.5) -- (0.5,0);
	\filldraw[thick, gray] (0.5,0.5) circle [radius=0.07];
	\draw (1,0.5) circle [radius=0.07];
	\draw (0.5,1) circle [radius=0.07];
	\draw[thick] (0.5,1) -- (1,0.5);
	\filldraw[thick, gray] (1,1) circle [radius=0.07];
	\draw (1,1.5) circle [radius=0.07];
	\draw (1.5,1) circle [radius=0.07];
	\draw[thick] (1,1.5) -- (1.5,1);
	\filldraw[thick, gray] (1.5,1.5) circle [radius=0.07];
	\draw (2,1.5) circle [radius=0.07];
	\draw (1.5,2) circle [radius=0.07];
	\draw[thick] (1.5,2) -- (2,1.5);
	\filldraw[thick, gray] (2,2) circle [radius=0.07];
	\draw (2,2.5) circle [radius=0.07];
	\draw (2.5,2) circle [radius=0.07];
	\draw[thick] (2,2.5) -- (2.5,2);
	\filldraw[thick, gray] (2.5,2.5) circle [radius=0.07];
	\draw (3,2.5) circle [radius=0.07];
	\draw (2.5,3) circle [radius=0.07];
	\draw[thick] (2.5,3) -- (3,2.5);
	\filldraw[thick, gray] (3,3) circle [radius=0.07];
	\draw (3,3.5) circle [radius=0.07];
	\draw (3.5,3) circle [radius=0.07];
	\draw[thick] (3,3.5) -- (3.5,3);
	\filldraw[thick, gray] (3.5,3.5) circle [radius=0.07];
	\draw (4,3.5) circle [radius=0.07];
	\draw (3.5,4) circle [radius=0.07];
	\draw[thick] (3.5,4) -- (4,3.5);
	\filldraw[thick, gray] (4,4) circle [radius=0.07];
	\draw (4,4.5) circle [radius=0.07];
	\draw (4.5,4) circle [radius=0.07];
	\draw[thick] (4,4.5) -- (4.5,4);
	\filldraw[thick, gray] (4.5,4.5) circle [radius=0.07];
	\draw (5,4.5) circle [radius=0.07];
	\draw (4.5,5) circle [radius=0.07];
	\draw[thick] (5,4.5) -- (4.5,5);
	\filldraw[thick, gray] (5,5) circle [radius=0.07];
	\end{tikzpicture}
	\caption{The gray dots are the points $(m,m)$ and the line segments connect pairs of points $(m-1,m)$ and $(m,m-1)$.  The black squares graph the sequence $b(m)$ for the tiling from Examples \ref{example1} and \ref{example2}.   Since $b(7)=8$ and $b(8)=7$, we know that rank of the tiling is 8.} \label{fig:M1}
\end{figure}
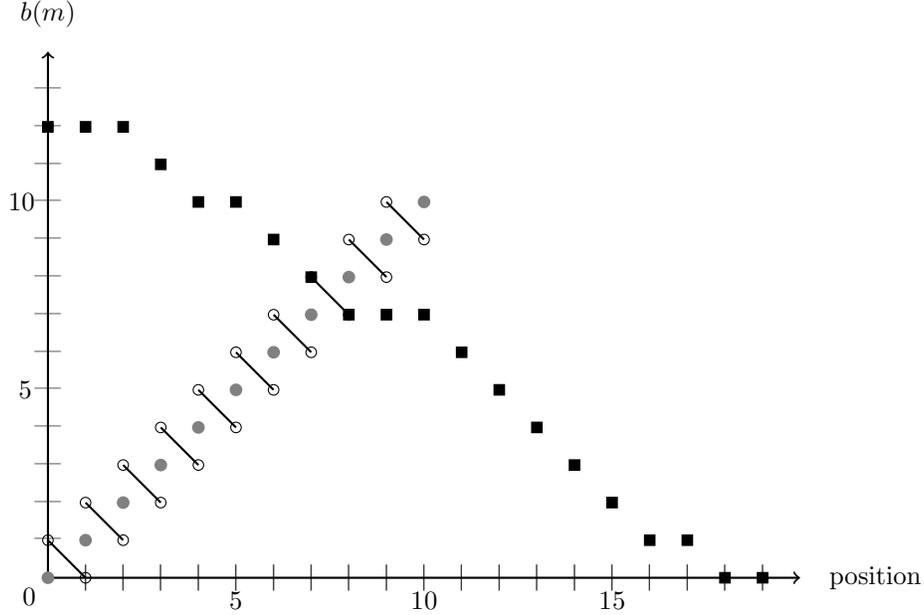


We are now ready to give a third expression for $F(z,q) = (-zq;q)_{\infty}$.

\begin{thm} 
	\label{GenEPNT}
	For $|q| < 1$,
	\begin{eqnarray}
	(-zq;q)_{\infty} &=& 1 + \sum_{m=1}^{\infty} \left (  \frac{(-zq;q)_{m-1}}{(q;q)_{m-1}}  z^{m} q^{m(3m-1)/2} + 
	\frac{(-zq;q)_{m}}{(q;q)_{m}}  z^{m} q^{m(3m+1)/2} \right ) \label{BestGenEPNT} \\
	&=& \sum_{m=0}^{\infty} 
	\frac{(-zq;q)_{m}}{(q;q)_{m}}  z^{m} q^{m(3m+1)/2}  (1+zq^{2m+1}). 
	\end{eqnarray}
\end{thm}

\begin{proof}

We proceed by 
writing down the generating function for all tilings of rank $m$,
and then summing over every possible rank of a tiling.
We do this by first considering the possible weights of the first $m$ positions of a tiling of rank $m$, and then generating the tiles at positions greater than $m$ by placing them in the first available positions and projecting them.


For tilings of rank $m \geq 1$,
notice that the weight of the tile at each position $i$ 
for $1 \leq i \leq m-1$ 
must be either $1$ or $zq^i$, and thus
the first $m-1$ positions are generated by 
$(-zq;q)_{m-1}$. 
The remainder of the generating function comes from two cases depending on whether $b(m) = m$ or $m-1$.
 
If $b(m) = m-1$, as we have seen above we also have $b(m-1) = m$, 
and so position $m$ must have a black square.
Thus the contribution to the generating function of position $m$ is $zq^m$.
For positions greater than $m$, we can generate all possibilities by first placing 
$m-1$ black tiles at locations $m+1$ through $2m-1$, and then projecting them in all possible ways.
The initial placement of those $m-1$ tiles has weight $z^{m-1} q^{3m(m-1)/2}$, and all possible projections are generated by $1/{(q;q)_{m-1}}$.
Thus rank $m$ tilings with 
$b(m) = m-1$ and $b(m-1) = m$ are generated by
\begin{equation} \label{rankmcase2}
(-zq;q)_{m-1} \times zq^m \times \frac{z^{m-1} q^{3m(m-1)/2}}{(q;q)_{m-1}}
 = \frac{(-zq;q)_{m-1}}{(q;q)_{m-1}}  z^{m} q^{m(3m-1)/2}.
\end{equation}

If $b(m) = m$, position $m$ may have a white square or a black square, and thus the contribution to the generating function of position $m$ is $1+zq^m$.
For positions greater than $m$, we can generate all possibilities by first placing 
$m$ black tiles at locations $m+1$ through $2m$, and then projecting them in all possible ways.
The initial placement of those $m$ tiles has weight $z^{m} q^{m(3m+1)/2}$, and all possible projections are generated by $1/{(q;q)_{m}}$.
Thus rank $m$ tilings with 
$b(m) = m$ are generated by
\begin{equation}  \label{rankmcase1}
 (-zq;q)_{m-1} \times (1+zq^m) \times \frac{z^{m} q^{m(3m+1)/2}}{(q;q)_{m}}
 = \frac{(-zq;q)_{m}}{(q;q)_{m}}  z^{m} q^{m(3m+1)/2}.
\end{equation}
Summing (\ref{rankmcase2}) and (\ref{rankmcase1}) over all $m \geq 1$ and adding $1$ for the weight of the tiling of rank $0$
(the tiling of all white tiles),
we have our new expression for the generating function of all tilings.

\end{proof}

Theorem \ref{GenEPNT} is originally due to Sylvester \cite[p.~281]{Sylv}.
Alladi has also given a very interesting combinatorial treatment of Theorem \ref{GenEPNT} in 
\cite{Alladi}.
Setting $z=-1$ in (\ref{BestGenEPNT}) creates substantial cancellation and recovers Euler's Pentagonal Number Theorem.

\begin{cor}[Euler's Pentagonal Number Theorem] \label{EPNT}
	\begin{eqnarray*}
		(q;q)_{\infty} &=& 
		 1 + \sum_{m=1}^{\infty} (-1)^{m} q^{m(3m-1)/2} (1+q^m).
	\end{eqnarray*}
\end{cor}

\section{The generalized rank}

In Section \ref{tileandrank}, we saw how the rank of a tiling provided a new perspective which allowed us to prove a general $q$-series identity which had not been proven previously using tiling and projection.
In this section, we generalize the notion of the rank, which leads to a two-parameter infinite family of 
identities for $F(z,q) = (-zq;q)_{\infty}$.

The search for a tiling proof of Euler's Pentagonal Number Theorem
is what led the authors to discover the original rank statistic.
To recover Euler's Pentagonal Number Theorem,
the rank strikes the perfect balance between 
the number of projectiles 
and how far to the right we go on the $1 \times \infty$ board to
place the projectiles in their initial configuration. 
However, any ratio (or even a non-constant ratio) between these two quantities will generate an identity for $F(z,q) = (-zq;q)_{\infty}$. 
Thus we now generalize the notion of the rank.

\begin{defn}
	The \emph{$(k,l)$-rank} of a tiling is the least $m$ such that $b(km) \leq lm$.
\end{defn}

Notice that the $(1,1)$ rank of a tiling is our original rank of a tiling given in Definition \ref{def:rank}.

\begin{ex} \label{example3}
	Consider the tiling with black squares at positions
	$3,4,6,7,8,11,12,13,14,15,16$, and $18$.
	For each position $m$, we compute $b(m)$ in Table \ref{btable2} below.
	Since $b(1 \cdot 5) = b(5) =  10 \leq 2 \cdot 5 $, we see that the $(1,2)$-rank of this tiling is $5$.
\end{ex}

\begin{table}[h]
	\begin{center}
		\begin{tabular}{ l|*{20}{c}r}
			\hline
			$m$ & 1 & 2 & 3 & 4 & 5 & 6 & 7 & 8 & 9 & 10 & 11 & 12 & 13 & 14 & 15 & 16 & 17 & 18 & 19 \\  \hline
			$b(m)$ & 12 & 12 & 11 & 10 & 10 & 9 & 8 & 7 & 7 & 7 & 6 & 5 & 4 & 3 & 2 & 1 & 1 & 0 & 0 \\ \hline
		\end{tabular}
	\end{center}
	\caption{The rank of this tiling is $8$, and the $(1,2)$-rank of this tiling is $5$.}
	\label{btable2}
\end{table}

As mentioned earlier, the reader may find it helpful to think of the graph $b(m)$ as a path along lattice points.
As this lattice path goes from the vertical axis down to the horizontal axis, 
it will intersect exactly one rectangular region of lattice points of the form 
$\{(km-j,\ell m - i) \mid 0 \leq j \leq k , 0 \leq i < \ell \}$ 
(see Figure \ref{fig:M2}).
The rank of a tiling is the value of $m$ corresponding to which rectangular region its lattice path intersects.

\begin{ex}
	To determine the $(4,3)$-rank of a tiling, we can graph the rectangular regions of lattice points of the form 
	$\{(4m-j,3 m - i) \mid 0 \leq j \leq 4 , 0 \leq i < 3 \}$
	and then plot the sequence $(m,b(m))$ to see which of the rectangular regions it intersects.
	In Figure \ref{fig:M2}, we have done this for the tiling from Examples \ref{example1}, \ref{example2}, and \ref{example3}.
	Since the sequence $(m,b(m))$ intersects the third rectangular region, the $(4,3)$-rank of the tiling is 3. 
\end{ex}

\begin{rmk} \label{klrankcases}
	The fact that $b(m)$ is a non-increasing sequence which never decreases by more than one when $m$ increases by one allowed us to algebraically characterize the original rank with just two mutually exclusive cases in Remark \ref{rankcases}.
	Since the sequence $b(km)$ 
	could decrease by as much as $k$ when $m$ increases by one,
	and the condition of being $\leq lm$ is more coarse than
	the condition of being $\leq m$,
	our two mutually exclusive cases to characterize the smallest $m$ such that $b(km) \leq lm$ each become more involved as $k$ and $\ell$ increase. 
	In fact, the rank of a tiling will be
	the unique $m$ such that either
	\begin{enumerate}
	\item $b(km) = \ell m - i$, for some  $0 \leq i \leq \ell-1$, or
		\item $b(km - j)=\ell m-\ell$ and $b(km-j-1)=\ell m -\ell +1$ for some $0 \leq j \leq k-1$.
\end{enumerate}

\end{rmk}


Remark \ref{klrankcases} reflects the fact that 
as the path leaves the unique rectangular region in intersects, 
either (1), the path will pass through a lattice point where 
 $b(km) = lm - i$, for some  $0 \leq i \leq \ell-1$ (along the right edge of the rectangle),
 or (2), it will pass through a pair of lattice points such that
 $b(km-j-1)=\ell m -\ell +1$ and $b(km - j)=\ell m-\ell$ for some $0 \leq j \leq k-1$ (exiting the bottom of the rectangle).
In Figure \ref{fig:M2} below,
we have indicated this by 
graphing all of the points in each rectangular region,
coloring the points $(km,lm-i)$ in gray,
and graphing the line segments connecting 
$(km-j-1,lm-l+1)$ and $(km-j,lm-l)$ for each $m$.
For every tiling, the corresponding graph of $b(m)$ will either intersect one of the gray points,
or it will traverse one of the line segments (but not both).

\begin{figure}[h]
	
	\centering
	
	\begin{tikzpicture}
	
	\draw  [<->, thick](0,7) -- (0,0) -- (10,0);		
	\node at (11,0) {position};		
	\node at (0,7.5) {$b(m)$};		
	\node at (-0.25,-0.25) {0};		
	\node at (0,0.5) {---};		
	\node at (0,1) {---};		
	\node at (0,1.5) {---};		
	\node at (0,2) {---};		
	\node at (0,2.5) {---};		
	\node at (-0.3,2.5) {5};		
	\node at (0,3) {---};		
	\node at (0,3.5) {---};		
	\node at (0,4) {---};		
	\node at (0,4.5) {---};		
	\node at (0,5) {---};		
	\node at (-0.35,5) {10};		
	\node at (0,5.5) {---};		
	\node at (0,6) {---};		
	\node at (0,6.5) {---};		
	\node at (0.5,0) {$\mid$};		
	\node at (1,0) {$\mid$};		
	\node at (1.5,0) {$\mid$};		
	\node at (2,0) {$\mid$};		
	\node at (2.5,0) {$\mid$};		
	\node at (2.5,-0.3) {5};		
	\node at (3,0) {$\mid$};		
	\node at (3.5,0) {$\mid$};		
	\node at (4,0) {$\mid$};		
	\node at (4.5,0) {$\mid$};		
	\node at (5,0) {$\mid$};		
	\node at (5,-0.3) {10};		
	\node at (5.5,0) {$\mid$};		
	\node at (6,0) {$\mid$};		
	\node at (6.5,0) {$\mid$};		
	\node at (7,0) {$\mid$};		
	\node at (7.5,0) {$\mid$};		
	\node at (7.5,-0.3) {15};		
	\node at (8,0) {$\mid$};		
	\node at (8.5,0) {$\mid$};		
	\node at (9,0) {$\mid$};		
	\node at (9.5,0) {$\mid$};

	\draw [fill] (-0.07,5.93) rectangle (0.07,6.07);
	\draw [fill] (0.43,5.93) rectangle (0.57,6.07);
	\draw [fill] (0.93,5.93) rectangle (1.07,6.07);
	\draw [fill] (1.43,5.43) rectangle (1.57,5.57);
	\draw [fill] (1.93,4.93) rectangle (2.07,5.07);
	\draw [fill] (2.43,4.93) rectangle (2.57,5.07);
	\draw [fill] (2.93,4.43) rectangle (3.07,4.57);
	\draw [fill] (3.43,3.93) rectangle (3.57,4.07);
	\draw [fill] (3.93,3.43) rectangle (4.07,3.57);
	\draw [fill] (4.43,3.43) rectangle (4.57,3.57);	
	\draw [fill] (4.93,3.43) rectangle (5.07,3.57);
	\draw [fill] (5.43,2.93) rectangle (5.57,3.07);
	\draw [fill] (5.93,2.43) rectangle (6.07,2.57);
	\draw [fill] (6.43,1.93) rectangle (6.57,2.07);
	\draw [fill] (6.93,1.43) rectangle (7.07,1.57);
	\draw [fill] (7.43,0.93) rectangle (7.57,1.07);
	\draw [fill] (7.93,0.43) rectangle (8.07,0.57);
	\draw [fill] (8.43,0.43) rectangle (8.57,0.57);
	\draw [fill] (8.93,-0.07) rectangle (9.07,0.07);
	\draw [fill] (9.43,-0.07) rectangle (9.57,0.07);

	\filldraw [thick, gray] (0,0) circle [radius=0.07];
	\draw (0,0.5) circle [radius=0.07];
	\draw [thick] (0,0.5) -- (0.5,0);
	\draw (0,1) circle [radius=0.07];
	\draw [thick] (0.5,0.5) -- (1,0);
	\draw (0,1.5) circle [radius=0.07];
	\draw [thick] (1,0.5) -- (1.5,0);
	\draw (0.5,0.5) circle [radius=0.07];
	\draw [thick] (1.5,0.5) -- (2,0);
	\draw (0.5,1) circle [radius=0.07];
	\draw [thick] (2,2) -- (2.5,1.5);
	\draw (0.5,1.5) circle [radius=0.07];
	\draw [thick] (2.5,2) -- (3,1.5);
	\draw (1,0.5) circle [radius=0.07];
	\draw [thick] (3,2) -- (3.5,1.5);
	\draw (1,1) circle [radius=0.07];
	\draw [thick] (3.5,2) -- (4,1.5);
	\draw (1,1.5) circle [radius=0.07];
	\draw [thick] (4,3.5) -- (4.5,3);
	\draw (1.5,0.5) circle [radius=0.07];
	\draw [thick] (4.5,3.5) -- (5,3);
	\draw (1.5,1) circle [radius=0.07];
	\draw [thick] (5,3.5) -- (5.5,3);
	\draw (1.5,1.5) circle [radius=0.07];
	\draw [thick] (5.5,3.5) -- (6,3);
	
	\filldraw [thick, gray] (2,0.5) circle [radius=0.07];
	\filldraw [thick, gray] (2,1) circle [radius=0.07];
	\filldraw [thick, gray] (2,1.5) circle [radius=0.07];
	\draw (2,2) circle [radius=0.07];
	\draw (2,2.5) circle [radius=0.07];
	\draw (2,3) circle [radius=0.07];
	\draw (2.5,2) circle [radius=0.07];		
	\draw (2.5,2.5) circle [radius=0.07];		
	\draw (2.5,3) circle [radius=0.07];		
	\draw (3,2) circle [radius=0.07];		
	\draw (3,2.5) circle [radius=0.07];		
	\draw (3,3) circle [radius=0.07];
	\draw (3.5,2) circle [radius=0.07];		
	\draw (3.5,2.5) circle [radius=0.07];		
	\draw (3.5,3) circle [radius=0.07];		
	\filldraw [thick, gray] (4,2) circle [radius=0.07];		
	\filldraw [thick, gray] (4,2.5) circle [radius=0.07];		
	\filldraw [thick, gray] (4,3) circle [radius=0.07];		
	\draw (4,3.5) circle [radius=0.07];		
	\draw (4,4) circle [radius=0.07];		
	\draw (4,4.5) circle [radius=0.07];		
	\draw (4.5,3.5) circle [radius=0.07];		
	\draw (4.5,4) circle [radius=0.07];		
	\draw (4.5,4.5) circle [radius=0.07];		
	\draw (5,3.5) circle [radius=0.07];		
	\draw (5,4) circle [radius=0.07];		
	\draw (5,4.5) circle [radius=0.07];		
	\draw (5.5,3.5) circle [radius=0.07];		
	\draw (5.5,4) circle [radius=0.07];		
	\draw (5.5,4.5) circle [radius=0.07];		
	\filldraw [thick, gray] (6,3.5) circle [radius=0.07];		
	\filldraw [thick, gray] (6,4) circle [radius=0.07];		
	\filldraw [thick, gray] (6,4.5) circle [radius=0.07];
	\draw [thick] (6,5) -- (6.5,4.5);
	\draw [thick] (6.5,5) -- (7,4.5);
	\draw [thick] (7,5) -- (7.5,4.5);
	\draw [thick] (7.5,5) -- (8,4.5);
	\draw (6,5) circle [radius=0.07];
	\draw (6,5.5) circle [radius=0.07];
	\draw (6,6) circle [radius=0.07];
	\draw (6.5,5) circle [radius=0.07];
	\draw (6.5,5.5) circle [radius=0.07];
	\draw (6.5,6) circle [radius=0.07];
	\draw (7,5) circle [radius=0.07];
	\draw (7,5.5) circle [radius=0.07];
	\draw (7,6) circle [radius=0.07];
	\draw (7.5,5) circle [radius=0.07];
	\draw (7.5,5.5) circle [radius=0.07];
	\draw (7.5,6) circle [radius=0.07];
	\filldraw [thick, gray] (8,5) circle [radius=0.07];
	\filldraw [thick, gray] (8,5.5) circle [radius=0.07];
	\filldraw [thick, gray] (8,6) circle [radius=0.07];
	
	\end{tikzpicture}
	
	\caption{The circular dots are the points in the rectangular regions of lattice points of the form 
		$\{(4m-j,3 m - i) \mid 0 \leq j \leq 4 , 0 \leq i < 3 \}$.  The gray dots are the points $(4m,3 m -i)$ for $0 \leq i \leq 2$, and the line segments connect pairs of points $(4m-j-1,3 m -2)$ and $(4m-j,3 m-3)$ for $0 \leq j \leq 3$. The black squares graph the sequence $b(m)$ for the tiling from Examples \ref{example1}, \ref{example2}, and \ref{example3}.  It is easy to see that the sequence $(m,b(m))$ passes through the third rectangular region, and thus the $(4,3)$-rank is $3$.  Algebraically, this is witnessed by the fact that $b(10) = b(4 \cdot 3 - 1 - 1) =  3 \cdot 3 -2 = 7$ and $b(11) = b(4 \cdot 3 -1) =  3 \cdot 3 -3 = 6$.} \label{fig:M2}
	
\end{figure}
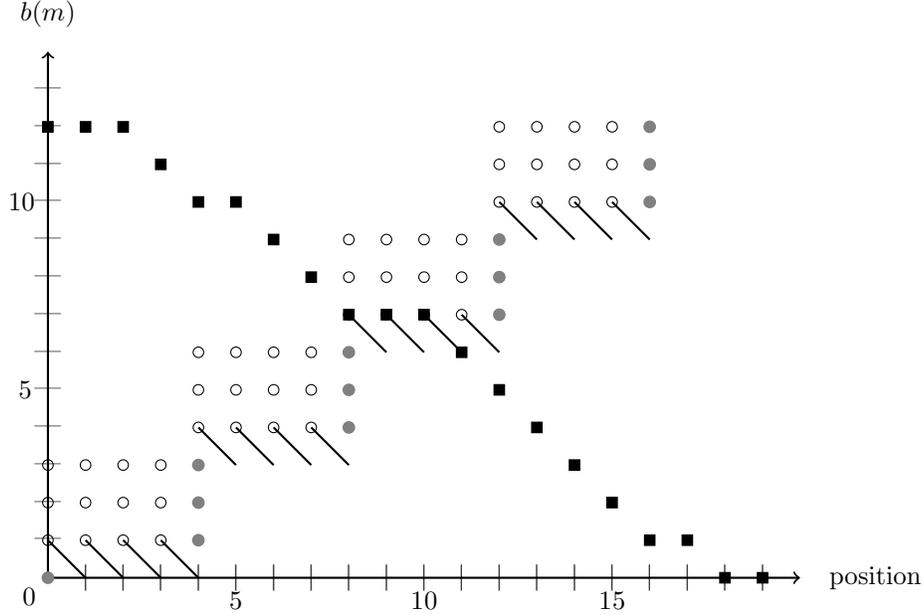


We are now ready to give an infinite family of identities for $F(z,q) = (-zq;q)_{\infty}$.

\begin{thm} \label{klthm}
	Let $|q| < 1$. For any positive integers $k$ and $\ell$, we have 
	\begin{align*}
	(-zq;q)_{\infty}=& 1 + \sum\limits_{m=1}^{\infty} \sum\limits_{i=0}^{\ell-1} 
	\frac{(-zq;q)_{km}}{(q;q)_{\ell m-i}}z^{\ell m-i}q^{(\ell m-i)((2k+\ell)m-i+1)/2} \\
	&+\sum\limits_{m=1}^{\infty} \sum\limits_{j=0}^{k-1}  \frac{(-zq;q)_{km-j-1}}{(q;q)_{\ell m-\ell}}z^{\ell m-\ell+1}q^{(\ell m - \ell +1)((2k+\ell)m-2j - \ell)/2}.
	\end{align*}
\end{thm}

\begin{proof}
We proceed by 
writing down the generating function for all tilings of $(k,\ell)$-rank $m$,
and then summing over every possible $(k,\ell)$-rank of a tiling.

For tilings of $(k,\ell)$-rank $m \geq 1$,
if we are in Case (1) from Remark \ref{klrankcases}, our only condition on the coloring of our tiles is that 
	$b(km) = \ell m - i$, for some  $0 \leq i \leq \ell-1$.
Notice that in this case the weight of the tile at each position $i$ 
for $1 \leq i \leq km$ 
can be either $1$ or $zq^i$, and thus
the first $km$ positions are generated by 
$(-zq;q)_{km}$. 
For positions greater than $m$, we can generate all possibilities by first placing 
$\ell m - i$ black tiles at locations $km+1$ through $km+ \ell m - i$, and then projecting them in all possible ways.
The initial placement of those $\ell m - i$ tiles has weight $z^{\ell m-i}q^{(\ell m-i)((2k+\ell)m-i+1)/2}$, and all possible projections are generated by $1/{(q;q)_{\ell m - i}}$.
Thus $(k,\ell)$-rank $m$ tilings with 
$b(km) = \ell m -i$ are generated by
\begin{equation}  \label{klrankmcase1}
\frac{(-zq;q)_{km}}{(q;q)_{\ell m - i}}  z^{\ell m-i}q^{(\ell m-i)((2k+\ell)m-i+1)/2}.
\end{equation}
If we are in Case (2) from Remark \ref{klrankcases}, our only condition on the coloring of our tiles is that 
 $b(km - j)=\ell m-\ell$ and $b(km-j-1)=\ell m -\ell +1$ for some $0 \leq j \leq k-1$.
Notice that in this case the weight of the tile at each position $i$ 
for $1 \leq i \leq km-j-1$ 
must be either $1$ or $zq^i$, and thus
the first $km-j-1$ positions are generated by 
$(-zq;q)_{km-j-1}$. 
Notice also that
position $km-j$ must have a black square,
and thus the contribution to the generating function of position $km-j$ is $zq^{km-j}$.
For positions greater than $km-j$, we can generate all possibilities by first placing 
$\ell m-\ell$ black tiles at locations $km-j+1$ through $km-j+\ell m - \ell$, and then projecting them in all possible ways.
The initial placement of those $\ell m-\ell$ tiles has weight 
$z^{\ell m-\ell} q^{(\ell m - \ell)(2km-2j+\ell m - \ell + 1)/2}$, and all possible projections are generated by $1/{(q;q)_{\ell m-\ell}}$.
Thus $(k,\ell)$-rank $m$ tilings with 
$b(km - j)=\ell m-\ell$ and $b(km-j-1)=\ell m -\ell +1$ are generated by
\begin{eqnarray} \label{klrankmcase2}
(-zq;q)_{km-j-1} &\times& zq^{km-j} \times \frac{z^{\ell m-\ell} q^{(\ell m - \ell)(2km-2j+\ell m - \ell + 1)/2}}{(q;q)_{\ell m-\ell}} \\
\notag &=& \frac{(-zq;q)_{km-j-1}}{(q;q)_{\ell m-\ell}}  z^{\ell m-\ell+1} q^{(\ell m - \ell +1)(2km-2j+\ell m - \ell + 1)/2}.
\end{eqnarray}
Summing (\ref{klrankmcase1}) and (\ref{klrankmcase2}) over all $m \geq 1$ and adding $1$ for the weight of the tiling of rank $0$
(the tiling of all white tiles),
we have our new expression for the generating function of all tilings.

\end{proof}

Theorem \ref{klthm} has many interesting corollaries.
Several of the most interesting corollaries come from setting $z=-1$
because of the cancellation that then occurs between the numerators and denominators of the summands.
Of course, since the $(1,1)$ rank of a tiling is just our original rank 
setting $k=\ell=1$ and $z=-1$ recovers
Euler's Pentagonal Number Theorem.

The next simplest corollary of Theorem \ref{klthm} with $z=-1$ comes from setting $k=2$ and $\ell=1$.

\begin{cor} \label{cor:hept}
	
	We have
	
	\begin{eqnarray*}
	(q;q)_{\infty} &=&	
	1+\sum_{m=1}^{\infty}(-1)^{m} \left \{  (q^{m+1};q)_{m} \ q^{m(5m+1)/2} +
	 (q^m;q)_{m} \ q^{m(5m-1)/2}
	+  (q^m;q)_{m-1} \ q^{m(5m-3)/2} \right \}	\\
	&=&	
	1+\sum_{m=1}^{\infty} (-1)^m \left \{  (q^{m+1};q)_{m-2} \ 
	(1-q^{3m-1}-q^{4m}+q^{6m-1}) \ q^{m(5m-3)/2} 
	 \right \}.
	\end{eqnarray*}
	
\end{cor}

\begin{rmk}
	Notice that the exponents $m(5m-3)/2$ are the heptagonal numbers,
	and thus Corollary \ref{cor:hept} is a ``heptagonal number theorem" of sorts,
	where we have expressed $(q;q)_{\infty}$ as a sum of heptagonal powers of $q$ multiplied by simple polynomials.
	Notice also that if we apply Euler's Pentagonal Number Theorem to the left-hand side of Corollary \ref{cor:hept}, we see that this is a way to construct the \emph{pentagonal} number series from simple polynomials multiplied \emph{heptagonal} powers of $q$. 
	If we set $z=-1$, $k=3$ and $\ell=1$ in Theorem \ref{klthm},
	we get a similar ``nonagonal number theorem,"
	and more generally, any instance of setting $z=-1$ and $\ell=1$ in Theorem \ref{klthm},
	leads to a ``$(2k+3)$-agonal number theorem."
\end{rmk}

Setting $z=1$ in Theorem \ref{klthm} also gives some interesting corollaries, although they are of a very different flavor than Corollaries \ref{EPNT} and \ref{cor:hept}.
Notice that with $z=1$, if we expand the denominators in the sum on the right-hand side of Theorem \ref{klthm} as geometric series, everything on both sides of the equation becomes positive.
In contrast to corollaries where $z=-1$ and cancellation plays a huge role,
when we let $z=1$, there is no cancellation whatsoever, and our results are purely additive.
The astute reader will also notice that in the case where $z=1$, while the left-hand side of Theorem \ref{klthm} is the generating function for partitions into distinct parts, the summands on the right-hand side are related to a variation of overpartitions.
We leave the exploration of this connection between partitions into distinct parts and a variation of overpartitions to the reader.


\begin{rmk}
In Remark \ref{klrankcases}, we characterized the $(k,l)$-rank of a tiling by observing that the sequence of points $(m,b(m))$ must pass through exactly one of the specified rectangular regions of lattice points.
We then constructed Cases (1) and (2) by observing that 
passing through a rectangular region occurs if and only if the sequence of points 
crosses the right edge or exits the bottom edge the rectangle.
However, there are many different sets of necessary and sufficient conditions for passing through a rectangular region.
For example, we could instead observe that the sequence of points $(m,b(m))$ must either cross the left edge (intersecting a point) or enter along the top edge (traversing a line segment) of the rectangle.
Alternately, we could construct many different configurations of $k$ line segments and $l$ points such that 
intersecting a rectangular region is equivalent to the sequence of points $(m,b(m))$ either passing through one of the $k$ points or traversing one of the $l$ line segments of the configuration.
Using a configuration of $k$ line segments and $l$ points other than the one we have used (along the bottom and right edges of the rectangle) would generate what appears to be a slightly different identity for $(-zq;q)$ than the one in Theorem \ref{klthm}.
However, it turns out that 
the difference is trivial enough that we do not consider these alternate identities to be true generalizations.
In particular,
the right-hand side of the equation in Theorem \ref{klthm} can be transformed into 
the right-hand side that would come from using any alternate configuration
by repeated applications of the simple algebraic identity
$$
(1+zq^x)q^y + (1-q^y)
=
1 + zq^{x+y},
$$
with varying values of $x$ and $y$
within each summand.
\end{rmk}

Setting $k=1$ in Theorem \ref{klthm}, we have
\begin{cor} \label{1lcor}
	Let $|q| < 1$. For any positive integer $\ell$, we have 
	\begin{align*}
	(-zq;q)_{\infty}=& 1 + \sum\limits_{m=1}^{\infty} \sum\limits_{i=0}^{\ell-1} 
	\frac{(-zq;q)_{m}}{(q;q)_{\ell m-i}}z^{\ell m-i}q^{(\ell m-i)((2+\ell)m-i+1)/2} \\
	&+\sum\limits_{m=1}^{\infty}  \frac{(-zq;q)_{m-1}}{(q;q)_{\ell m-\ell}}z^{\ell m-\ell+1}q^{(\ell m - \ell +1)((2+\ell)m - \ell)/2}.
	\end{align*}
\end{cor}

\section{Conclusion}

By introducing the rank, we were able to extend the scope of the method of tiling
to include Euler's Pentagonal Number Theorem and a three-parameter generalization. 
This generalization has many interesting corollaries, including a result of Sylvester, a ``$(2k+3)$-agonal number theorem," and identities relating partitions into distinct part to variations of overpartitions.
In fact, there are several more possible generalizations that we have not listed above. 
In \cite{LS1} and \cite{LS2}, Little and Sellers treat tilings with dominoes instead of squares quite thoroughly.
Introducing the concept of the rank into domino tilings reveals many more identities, all involving the $q$-Fibonacci polynomials of Carlitz \cite{Carlitz}.
These identities are of a somewhat different and more complicated form than most classical identities. 
Even without dominoes, there are other generalizations that are even more complicated. For example, for any two increasing sequences of nonnegative integers $X = \{x_m\}$ and $Y = \{y_m\}$, we can define the $(X,Y)$-rank to be least $m$ such that $b(x_m) \leq y_m$.
Writing down the generating function for all tilings of $(X,Y)$-rank $m$ and then summing over every possible $(X,Y)$-rank of a tiling gives yet another new identity for $(-zq;q)_{\infty}$ in each case, although those identities can be essentially arbitrarily complicated.
It is interesting that these generalizations exists, although it is unclear to the authors if any individual one of these generalizations is of any interest. 

An open problem that remains is whether or not Jacobi's Triple Product Identity can be proven using the method of tiling through the introduction of a rank-like statistic. 
In addition, it will be interesting to see if the scope of the method of tiling can be extended even further to include other families of identities by using the concept of the rank. 

Finally, the authors would like to thank George Andrews for feedback on Theorem \ref{klthm} and Nathan Kaplan for comments on an earlier version of this paper.



\begin{thebibliography}{9}
	


	\bibitem{Alladi}
Alladi, Krishnaswami
\emph{Partition identities involving gaps and weights.}
Trans. Amer. Math. Soc.
\textbf{349} (1997), no. 12, 5001--5019.

\bibitem{Carlitz} Carlitz, L.,
	\emph{Fibonacci notes. IV. q-Fibonacci polynomials.}
Fibonacci Quart.
\textbf{13}, (1975), 97--102.

	\bibitem{Euler}	Euler, L. {Introductio in analysin infinitorum.} Marcum-Michaelum Bousquet, Lausannae (1748).
	
	
		\bibitem{LS1}
	Little, David P. and Sellers, James A.
	\emph{New proofs of identities of {L}ebesgue and {G}\"ollnitz via
		tilings.}
	J. Combin. Theory Ser. A
	\textbf{116},
	(2009),
	no. 1,
	223--231.
	
	\bibitem{LS2}	Little, David P. and Sellers, James A.	\emph{A tiling approach to eight identities of {R}ogers.}	European J. Combin.	\textbf{31}	(2010),	no. 3,	694--709.
	
	
	\bibitem{Sylv}	Sylvester, J. J. 	\emph{A Constructive theory of Partitions, arranged in three Acts, an Interact and an Exodion.}	American J. Math	\textbf{5}	(1882),	251--330.
	
	
\end{thebibliography}
\end{document}